\documentclass{iopart}

\usepackage{amssymb}
\usepackage{amsthm, amsopn}

\usepackage[abbr]{harvard}

\newcommand{\R}{\mathbb{R}}
\newcommand{\N}{\mathbb{N}}

\newcommand{\Rb}{\overline{\R}}

\newcommand{\id}{\textnormal{Id}}

\DeclareMathOperator*{\argmin}{\textnormal{argmin}}

\newcommand{\supp}{\textnormal{supp}}

\newcommand{\set}[1]{\left\{ #1 \right\}}
\newcommand{\abs}[1]{\left| #1 \right|}
\newcommand{\norm}[1]{\left\| #1 \right\|}
\newcommand{\inner}[2]{\left\langle #1, #2 \right\rangle}

\newcommand{\bigo}{\mathcal{O}}
\newcommand{\ra}{\rightarrow}

\newcommand{\spa}{\textnormal{span}}

\theoremstyle{plain}

\newtheorem{thm}{Theorem}[section]

\newtheorem{ass}[thm]{Assumption}

\theoremstyle{definition}
\newtheorem{rem}[thm]{Remark}

\newtheorem{lem}[thm]{Lemma}

\newtheorem{cor}[thm]{Corollary}

\newtheorem*{example*}{Example}

\newtheorem*{dfn*}{Definition}
\newtheorem*{alg*}{Algorithm}

\theoremstyle{remark}

\makeatletter
\newcommand{\logmessage}[1]{\@latex@warning{#1}}
\makeatother

\begin{document}

\author{K. Frick$^1$ and M. Grasmair$^2$}  
\address{$^1$Institute for Mathematical Stochastics\\
University of G{\"o}ttingen\\
Goldschmidtstra{\ss}e 7, 37077 G{\"o}ttingen, Germany}
 
\address{$^2$Computational Science Center\\
University of Vienna\\ 
Nordbergstra{\ss}e 15, 1090 Vienna, Austria}

\eads{\mailto{frick@math.uni-goettingen.de},
\mailto{markus.grasmair@univie.ac.at}}

\title[Variational Inequalities and the Augmented Lagrangian
Method]{Regularization of Linear Ill-posed Problems by the Augmented Lagrangian
Method and Variational Inequalities}

\begin{abstract} 
  We study the application of the Augmented Lagrangian Method 
  to the solution of linear ill-posed problems.
  Previously, linear convergence rates with respect to the 
  Bregman distance have been derived under the classical
  assumption of a standard source condition.
  Using the method of variational inequalities,
  we extend these results in this paper
  to convergence rates of lower order,
  both for the case of an a priori parameter choice
  and an a posteriori choice based on Morozov's discrepancy principle.
  In addition, our approach allows the derivation
  of convergence rates with respect to distance measures
  different from the Bregman distance.
  As a particular application, we consider sparsity
  promoting regularization, where we derive a range of convergence
  rates with respect to the norm
  under the assumption of restricted injectivity in conjunction
  with generalized source conditions of H\"older type.
\end{abstract}
 

 

\section{Introduction}

We aim for the solution of the problem
\begin{equation}\label{intro:primal}
\inf_{u\in X} J(u)\quad\textnormal{ s.t. }\quad Ku = g,
\end{equation}
where $K\colon X\ra H$ is a linear and bounded mapping between a Banach space $X$
and a Hilbert space $H$ and where $J\colon X\ra\Rb$ is convex and lower
semi-continuous. We are particularly interested in the case when the right hand side in the linear
constraint is not at hand but only an approximation $g^\delta$ such that
\begin{equation}\label{intro:error}
\norm{g-g^\delta}\leq \delta
\end{equation}
for some $\delta>0$. A possible method
for computing a stable approximation of solutions of \eref{intro:primal} is the
\emph{augmented Lagrangian method (ALM)}, an iterative method that, for a given
initial value $p_0^\delta\in H$ and for $k=1,2,\ldots$, computes

\numparts
\begin{eqnarray}\label{intro:alm}  
u_k^\delta & \in & \argmin_{u\in X}\biggl[ \frac{\tau_k}{2}\norm{Ku - g^\delta}^2 +
J(u) - \inner{p_{k-1}^\delta}{Ku - g^\delta}\biggr]\label{alm:primal} \\
p_k^\delta & = & p_{k-1}^\delta + \tau_k(g^\delta - K
u_k^\delta)\label{alm:dual}.
\end{eqnarray}
\endnumparts
Here, $\set{\tau_k}_{k\in\N}$ denotes a pre-defined sequence of positive
parameters such that
\[
t_n := \sum_{k=1}^n \tau_k\ra\infty\quad\textnormal{ as }\quad n\ra\infty.
\]
The ALM was originally introduced in \cite{Hes69,Pow69} (under the name
\emph{method of multipliers}) as a solution method for problems of type
\eref{intro:primal} with \emph{exact} right hand side $g$. Since then, the ALM
was developed further in various directions; see e.g.\ \cite{ForGlo83,ItoKun08}
and the references therein. 
 
In the context of inverse problems, the ALM was first considered  for the
special case when $X$ is a Hilbert space and $J$ is a \emph{quadratic}
functional, i.e., $J(u) = {1\over 2}\norm{Lu}^2$ for a densely defined and closed
linear operator $L\colon D(L)\subset X\ra \tilde H$, where $\tilde H$ is some further
Hilbert space (here we set $J(u) = +\infty$ if $u\not\in D(L)$). For this
special case, it is readily seen that the ALM can be rewritten into
\begin{equation}\label{intro:itertik}
u_k^\delta = \argmin_{u\in X}\biggl[ \tau_k \norm{Ku - g^\delta}^2 +
\norm{L(u-u_{k-1}^\delta)}^2_{\tilde H}\biggr].
\end{equation}
The analysis of iteration \eref{intro:itertik} dates back to the papers
\cite{Kra60,Krj73}. The case when $L\equiv \id$ is referred to as the
\emph{iterated Tikhonov method} and has been studied in
\cite{Lar75,BriSch87,HanGro98,EngHanNeu96}. The regularization scheme that
results for $K\equiv\id$ is termed \emph{iterated Tikhonov--Morozov method} and
amounts to stably evaluate the (possibly unbounded) operator $L$ at $g$ given
only an approximation $g^\delta$ that satisfies \eref{intro:error}. For 
detailed analysis see e.g.~\cite{GroSch00,Gro07}.

A generalization of the iteration in \eref{intro:itertik} for total-variation
based image reconstruction has been established in \cite{OshBurGolXuYin05}
under the name \emph{Bregman iteration} and convergence properties were studied
in \cite{BurResHe07}. In \cite{FriSch10} it was pointed out that the Bregman
iteration and the iterated Tikhonov(--Morozov) method are special instances of
the ALM as it is stated in \eref{intro:alm}, and an improved convergence
analysis was developed. In \cite{FriLorRes11}, Morozov's discrepancy principle
\cite{Mor67} was studied for the ALM. The application of the ALM for the
regularization of nonlinear operators has been considered in
\cite{BacBur09,JunResVes11}.

Up to now, convergence rates for the ALM (in the context of inverse problems)
have only been derived under the assumption that the solutions $u^\dagger$ of
\eref{intro:primal} satisfy the \emph{standard source condition} \cite{BurOsh04}
\begin{equation}\label{intro:sc}
K^*p^\dagger \in \partial J(u^\dagger)\quad \textnormal{ for some }p^\dagger
\in H.
\end{equation}  
Here $K^*\colon H \to X^*$ denotes the adjoint operator of $K$ and $\partial J(u^\dagger)$ is
the subdifferential of $J$ at $u^\dagger$. This
typically results in a convergence rate of $\delta$ with respect to the Bregman
distance (for a definition of the subdifferential and the Bregman distance, see
Section \ref{dual}). In this paper we will extend these results to convergence  
rates of lower order by replacing \eref{intro:sc} by \emph{variational
inequalities}. The analysis will apply for both a priori and a posteriori
parameter selection rules, where the latter will be realized by Morozov's
discrepancy principle. In addition, our approach allows the derivation of
convergence rates with respect to distance measures different from the Bregman
distance.

The paper is organized as follows: In Section \ref{dual} we state basic
assumptions and review tools from convex analysis that are essential for our
analysis. In Section \ref{rates} we establish variational inequalities and
prove that these are sufficent for lower order convergence rates for the ALM
with suitable a priori stopping rules. In Section \ref{morozov} we reprove the
same convergence rates when Morozov's discrepancy principle is employed as an
a posteriori stopping rule. In Section \ref{ex} we finally consider some
examples that clarify the connection of the variational inequalities in Section
\ref{rates} and more classic notions of source conditions, such as the standard
source condition \eref{intro:sc} or H{\"o}lder-type conditions. Moreover, we
show for the particular scenario of sparsity promoting regularization how our
approach can be used to derive convergence rates with respect to the norm.

\section{Assumptions and Mathematical Prerequisites}\label{dual}

In this section we fix some basic assumptions as well as review basic notions
and facts from convex analysis. We start by delimiting minimal
functional analytic requirements.
\begin{ass}\label{intro:mainass}
\begin{enumerate} 
\item $X$ is a separable Banach space with topological dual $X^*$. We denote the
duality pairing of $X$ and $X^*$ by $\inner{\xi}{x}_{X^*, X} = \xi(x)$.
\item The operator $K\colon X\ra H$ is linear and continuous.
\item The functional $J\colon X\ra \Rb:=\R\cup\{+\infty\}$ is convex, lower semicontinuous and proper
with nonempty domain $D(J) = \set{u\in X~:~ J(u)<\infty}$.
\item  For each $g\in H$ and $c>0$ the set
\begin{equation*}
\Lambda(g,	c) = \set{u\in X~:~ \norm{Ku - g}^2 +  J(u)\leq c}
\end{equation*} 
is sequentially weakly pre-compact in $X$.
\end{enumerate}
\end{ass}

For our analysis we will make extensive use of tools from convex analysis
(here, we refer to \cite{EkeTem76} as a standard reference). We will henceforth
denote by $\partial J(u_0)$ the subdifferential of $J$ at $u_0\in X$, i.e., the set of all $\xi \in X^*$ such that
\begin{equation*}
J(u)\geq J(u_0) + \inner{\xi}{u-u_0}_{X^*,X},\quad \textnormal{ for all }u\in X.
\end{equation*}   
In this case, we call $\xi$ a subgradient of $J$ at $u_0$. We denote by
$K^*\colon H\ra X^*$ the adjoint operator of $K$, where we identify the
Hilbert space $H$ with its dual $H^*$ by means of Riesz' representation theorem.
Under Assumption \ref{intro:mainass}
it is guaranteed that solutions of \eref{intro:primal} exist for all $g\in
K(D(J))$ and that the iteration \eref{intro:alm} is well defined. The proof is
analogous to \cite[Lem. 3.1]{FriSch10}. 

Recall that the Legendre-Fenchel conjugate $J^*\colon X^*\ra \Rb$ of $J$ is
defined by $J^*(x^*) = \sup_{x\in X} \inner{x^*}{x}_{X^*,X} - J(x)$. The \emph{dual problem}
to \eref{intro:primal} is then defined by
\begin{equation}\label{dual:dual}
\inf_{p\in H} \Bigl[J^*(K^*p) - \inner{p}{g}\Bigr].
\end{equation}
Sufficient and necessary conditions for guaranteeing the existence of a solution
$u^\dagger\in X$ of \eref{intro:primal} and a solution $p^\dagger\in H$ of
\eref{dual:dual} are the \emph{Karush-Kuhn-Tucker} conditions, which read as
\begin{equation}\label{dual:kkt}
K^*p^\dagger \in \partial J(u^\dagger)\quad\textnormal{ and }\quad Ku^\dagger = g.
\end{equation}
From an inverse problems perspective, these conditions are understood as
\emph{source conditions} \cite{BurOsh04} that delimit a class of particular
regular solutions $u^\dagger$ of \eref{intro:primal} that can be reconstructed
from noisy data at a certain rate depending on the noise level $\delta$. If the
source condition \eref{dual:kkt} does not hold, then solutions of
\eref{intro:primal} may still exist (e.g. if Assumption \ref{intro:mainass}
holds) whereas \eref{dual:dual} has no solutions. The value of
\eref{dual:dual}, though, will still be finite:

\begin{lem}\label{dual:values}
Assume that Assumption \ref{intro:mainass} holds and let $u^\dagger\in X$ be a
solution of \eref{intro:primal}. Then
\begin{equation*}
\inf_{p\in H} \Bigl[J^*(K^*p) - \inner{p}{g}\Bigr] = -J(u^\dagger).
\end{equation*}
\end{lem}
\begin{proof}
Define a function $\Gamma\colon X\times H \ra \Rb$ by setting $\Gamma(u,p) = J(u)$ if
$Ku = g+p$ and $G(u,p) = +\infty$ else. 
According to \cite[Chap III. Prop. 2.1]{EkeTem76} the assertion holds, if the
function $p\mapsto h(p) = \inf_{u\in X} \Gamma(u,p)$ is finite and lower semicontinuous
at $p = 0$. Since $p(0) = J(u^\dagger)<\infty$ it remains to prove lower
semicontinuity. Let therefore $\set{p_k}_{k\in\N}$ be a sequence in $H$ such that
$p_k\ra 0$. Without loss of generality, we may, after possibly passing
to a subsequence, assume that $h(p_k) < \infty$ for every $k$,
which amounts to saying that the equation $Ku = g+p_k$ has a solution
$u_k \in X$ satisfying $J(u_k) < \infty$.
In addition, because of Assumption~\ref{intro:mainass},
we can choose $u_k$ such that the infimum in the definition of $h$ is
realized at $u_k$, that is, $h(p_k) = \Gamma(u_k,p_k)$.

Now, if $J(u_k)\ra\infty$ as $k\ra\infty$, nothing remains to be
proven. Thus we can assume that there exists a subsequence of
$\set{u_{k'}}$ such that $\sup_{k'\in\N} J(u_{k'}) < \infty$. It is not
restrictive to assume that
$\lim_{{k'}\ra\infty} J(u_{k'}) = \liminf_{k\ra\infty} J(u_k)$.
Moreover, we observe that
$\norm{Ku_k -g}^2 = \norm{p_k}^2$ is bounded, since $p_k\ra 0$. Thus it
follows from Assumption \ref{intro:mainass} that there
exists a further subsequence $\set{u_{k''}}$ such that $u_{k''}\rightharpoonup
\hat u$ for some $\hat u\in X$. This implies that $Ku_{k''}\rightharpoonup K\hat
u = g$, and the lower semicontinuity and convexity of $J$ finally proves that
\begin{equation*}
\fl
\liminf_{k\ra \infty} h(p_{k}) = \lim_{k'\ra\infty} J(u_{k'}) =
\liminf_{k''\ra\infty} J(u_{k''})\geq J(\hat u)\geq J(u^\dagger) = h(0). 
\end{equation*}
\end{proof}

Similar to the duality relation between the optimization problems
\eref{intro:primal} and \eref{dual:dual} such a relation can be established
for the ALM: As it was first observed in \cite{Roc74}, the dual sequence
$\set{p_0^\delta,p_1^\delta,\ldots}$ generated by the ALM can be characterized
by the \emph{proximal point method (PPM)}. To be more precise, for all $k\geq 1$,
\begin{equation}\label{dual:ppm}
p_k^\delta = \argmin_{p\in H}\biggl[\frac{1}{2}\norm{p - p_{k-1}^\delta}^2 +
\tau_k\left(J^*(K^*p) - \inner{p}{g^\delta}\right)\biggr].
\end{equation}
The PPM was introduced by Martinet in \cite{Mar70} for minimizing a convex
functional, which in the present situation is the dual functional
\eref{dual:dual}. The sequence $\set{p_k^\delta}$ generated by the PPM is known
to converge weakly to a solution of \eref{dual:dual} if it exists, i.e., when
\eref{dual:kkt} holds. If this is not the case, then still
$J^*(K^*p_k^\delta)-\inner{p_k^\delta}{g^\delta}$ converges to the value of the
program \eref{dual:dual} which, in the general case, may be $-\infty$, of course. 

\section{Convergence Rates}\label{rates}

It is well known that linear convergence rates (with respect to the Bregman
distance) for iterates of the ALM can be proven if the source condition
\eref{dual:kkt} holds (cf.\ \cite{BurResHe07,FriSch10}). In this section we
prove lower order rates of convergence in the case, when the source condition
\eref{dual:kkt} does not hold. Instead, we impose weaker regularity
conditions on solutions $u^\dagger$ of \eref{intro:primal} in terms of
\emph{variational inequalities}. We formulate this in the following

\begin{ass}\label{rates:varineq}
We are given an \emph{index function} $\Phi\colon[0,\infty)\ra [0,\infty)$,
i.e., a non-negative continuous function that is strictly increasing and concave with
$\Phi(0) = 0$. Moreover, $D\colon X\times X\ra [0,\infty]$ satisfies
$D(u,u) = 0$ whenever $u \in X$, and $u^\dagger$ is a
solution of \eref{intro:primal} is such that
\begin{equation}\label{rates:varineqcond}
D(u,u^\dagger)\leq J(u) - J(u^\dagger) + \Phi(\norm{Ku - g}^2)\quad\textit{ for
all }u\in X.
\end{equation}
We denote by $\Psi$ the Legendre-Fenchel conjugate of $\Phi^{-1}$.
\end{ass}
    
A typical choice is $D(u,v) = \beta D_J^\xi(u,v)$, where $\beta\in (0,1]$ and 
\begin{equation}\label{rates:bregman}
D^{\xi}_J(v,u) = J(v) - J(u) - \inner{\xi}{v-u}_{X^*,X}
\end{equation}
is the \emph{Bregman-distance} of $u$ and $v$ w.r.t. $\xi\in\partial J(v)$. With
this, \eref{rates:varineqcond} is equivalent to the condition
\begin{equation}\label{rates:varineqbreg}
\inner{\xi^\dagger}{u^\dagger-u}_{X^*,X}\leq
(1-\beta) D_J^{\xi^\dagger}(u,u^\dagger) + \Phi(\norm{Ku-g}^2)
\end{equation} 
for all $u\in X$. In this form, variational inequalities have been introduced in
\cite{HofKalPoeSch07,SchGraGroHalLen09} with $\Phi(s) = \sqrt{s}$,
and for general index functions in~\cite{BotHof10,Gra10b}. 

The following theorem asserts that the condition \eref{rates:varineqcond} in
Assumption \ref{rates:varineq} imposes sufficient smoothness on the true solution
$u^\dagger$ that the iterates of the ALM approach $u^\dagger$ with a
certain rate (that depends on $\Phi$).

\begin{thm}\label{rates:mainthm}
Let Assumptions \ref{intro:mainass} and \ref{rates:varineq} hold. Then, there exists a constant
$C>0$ such that 
\begin{equation*}
D(u_n^\delta, u^\dagger)\leq C t_n \left(\Psi\left(\frac{16}{t_n} \right)+
\delta^2\right)
\end{equation*}
and   
\begin{equation*}
\norm{Ku_n^\delta - g^\delta}^2 \leq C
\left(\Psi\left(\frac{16}{t_n}\right) + \delta^2\right).
\end{equation*}
In particular, if $t_n \asymp {1\over \Psi^{-1}(\delta^2)}$, then
\begin{equation*}
D(u_n^\delta, u^\dagger) = \bigo\left(
\frac{\delta^2}{\Psi^{-1}(\delta^2)}\right)\quad\textnormal{ and }\quad
\norm{Ku_n^\delta - g}^2 = \bigo(\delta^2).
\end{equation*}
\end{thm}

Theorem \ref{rates:mainthm} is a consequence of the following two Lemmas.

\begin{lem}\label{rates:genest}
Let Assumptions \ref{intro:mainass} and \ref{rates:varineq} hold and define for
$p\in H$, $t>0$ and $\delta \geq 0$
\begin{equation*}
\psi(p,t,\delta) = \left(t \Psi(16 \slash t) + t
\delta^2 + J^*(K^*p) + J(u^\dagger) - \inner{p}{g} + \frac{\norm{p}^2}{2t}
\right).
\end{equation*}
Then, there exists a constant $C>0$ such that
\begin{equation}\label{rates:genesteqn}
D(u_n^\delta, u^\dagger)\leq C \psi(p,t_n,\delta)\quad\textnormal{ and }\quad
\norm{Ku_n^\delta - g^\delta}^2 \leq C \frac{\psi(p,t_n,\delta)}{t_n}
\end{equation}
for all $p\in H$.
\end{lem}

\begin{proof}
Without loss of generality we assume that $p_0^\delta = 0$ and we shall agree
upon $G(p,g) = J^*(K^*p) - \inner{p}{g}$. In \cite[Lem.~2.1]{Gue91} it was proved
that for all $p\in V$
\begin{equation}\label{gueler}
\frac{t_n \norm{p_n^\delta - p_{n-1}^\delta}^2}{2\tau_n^2} \leq G(p,g^\delta)
- G(p_n^\delta,g^\delta) - \frac{\norm{p-p_n^\delta}^2}{2t_n} +
\frac{\norm{p}^2}{2t_n}.
\end{equation}
Since $G(p,g^\delta) - G(p_n^\delta,g^\delta) = G(p,g) - G(p_n^\delta,g) +
\inner{p-p_n^\delta}{g-g^\delta}$ and $p_n^\delta - p_{n-1}^\delta =
\tau_n(g^\delta - Ku_n^\delta)$, this implies that
\begin{eqnarray}
\fl\frac{t_n}{2} \norm{Ku_n^\delta - g^\delta}^2 & \leq G(p,g)
- G(p_n^\delta,g) - \frac{\norm{p-p_n^\delta}^2}{2t_n} +
\frac{\norm{p}^2}{2t_n} + \inner{p-p_n^\delta}{g-g^\delta} \nonumber\\
& \leq G(p,g) + J(u^\dagger) - \frac{\norm{p-p_n^\delta}^2}{2t_n} +
\frac{\norm{p}^2}{2t_n} + \inner{p-p_n^\delta}{g-g^\delta},\label{rates:aux0}
\end{eqnarray}
where the second inequality follows from Lemma \ref{dual:values}. Setting $p =
p_n^\delta$, this proves that
\begin{equation*}
\frac{t_n}{2}\norm{Ku_n^\delta - g^\delta}^2 \leq J^*(K^*p_n^\delta) -
\inner{p_n^\delta}{g} + J(u^\dagger)+ \frac{\norm{p_n^\delta}^2}{2t_n}.
\end{equation*}
Since $K^*p_n^\delta \in \partial J(u_n^\delta)$, we observe that
$J^*(K^*p_n^\delta) + J(u_n^\delta) = \inner{K^*p_n^\delta}{u_n^\delta}$ and
conclude that 
\begin{eqnarray*}
\fl\frac{t_n}{2}\norm{Ku_n^\delta - g^\delta}^2 & \leq  J(u^\dagger) -
J(u_n^\delta) + \inner{p_n^\delta}{Ku_n^\delta - g} +
\frac{\norm{p_n^\delta}^2}{2t_n} \\
& = J(u^\dagger) -J(u_n^\delta) + \inner{p_n^\delta}{Ku_n^\delta - g^\delta}
+\inner{p_n^\delta}{g^\delta - g}  + \frac{\norm{p_n^\delta}^2}{2t_n}.
\end{eqnarray*}
Applying Young's inequality $\inner{a}{b} \le \norm{a}^2\slash 2 + \norm{b}^2\slash 2$
first with $a = \sqrt{2\slash t_n}p_n^\delta$ and
$b = (Ku_n^\delta-g^\delta)\sqrt{t_n\slash 2}$,
and then with $a = p_n^\delta\slash\sqrt{t_n}$ and $b = \sqrt{t_n}(g^\delta-g)$,
we obtain
\begin{eqnarray*}
\frac{t_n}{4}\norm{Ku_n^\delta - g^\delta}^2 & \leq  J(u^\dagger) -
J(u_n^\delta)+\inner{p_n^\delta}{g^\delta - g} +
\frac{3\norm{p_n^\delta}^2}{2t_n} \\
& \leq J(u^\dagger) - J(u_n^\delta)+\frac{\delta^2 t_n}{2} +
\frac{2\norm{p_n^\delta}^2}{t_n} 
\end{eqnarray*}
Summarizing, we find that 
\begin{equation*}
\norm{Ku_n^\delta - g^\delta}^2 \leq \frac{4}{t_n}\left(J(u^\dagger) -
J(u_n^\delta)\right) + 2\delta^2 + \frac{8\norm{p_n^\delta}^2}{t_n^2}. 
\end{equation*}
Now, we observe from \eref{rates:varineqcond} that $J(u^\dagger) -
J(u_n^\delta) \leq - D(u_n^\delta, u^\dagger) + \Phi(\norm{Ku_n^\delta -
g}^2)$. Plugging this inequality into the above estimate yields
\begin{equation}\label{rates:aux1}
\fl
\norm{Ku_n^\delta - g^\delta}^2 + \frac{4}{t_n}D(u_n^\delta,
u^\dagger) \leq \frac{4}{t_n} \Phi(\norm{K u_n^\delta - g}^2) + 2\delta^2
+ \frac{8\norm{p_n^\delta}^2}{t_n^2}.
\end{equation}
Since $\Psi$ is the Legendre-Fenchel conjugate of $t\mapsto \Phi^{-1}(t)$,
i.e., $\Psi(s) = \sup_{t\geq 0} st - \Phi^{-1}(t)$, it follows that $st
\leq \Psi(s) + \Phi^{-1}(t)$ for all $s,t\geq 0$, and in particular, for $t =
\Phi(r)$, that $s\Phi(r)\leq \Psi(s) + r$ for all $s,r\geq 0$. Setting $s =
16\slash t_n$  and $r = \norm{Ku_n^\delta -g^\delta}^2$ gives
\begin{eqnarray*}
\frac{4}{t_n} \Phi(\norm{Ku_n^\delta - g}^2) & =
\frac{1}{4}\frac{16}{t_n} \Phi(\norm{Ku_n^\delta - g}^2) \\
& \leq \frac{1}{4} \Psi\left(\frac{16}{t_n}\right) + \frac{1}{4}
\norm{Ku_n^\delta - g}^2 \\
& \leq \frac{1}{4} \Psi\left(\frac{16}{t_n}\right) + \frac{1}{2}
\norm{Ku_n^\delta - g^\delta}^2 + \frac{\delta^2}{2}.
\end{eqnarray*}
Combining this with \eref{rates:aux1} yields
\begin{equation}\label{rates:aux2}
\frac{1}{2}\norm{Ku_n^\delta - g^\delta}^2
+ \frac{4}{t_n}D(u_n^\delta, u^\dagger) \leq\frac{1}{4}
\Psi\left(\frac{16}{t_n}\right) + \frac{5\delta^2}{2} + \frac{8\norm{p_n^\delta}^2}{t_n^2}.
\end{equation}
Finally, we observe again from \eref{gueler} that for all $p\in H$
\begin{eqnarray*}
\frac{\norm{p-p_n^\delta}^2}{2t_n^2} & \leq \frac{G(p,g^\delta) - G(p_n^\delta,
g^\delta)}{t_n} + \frac{\norm{p}^2}{2t_n^2} \\
&\leq \frac{G(p,g) - G(p_n^\delta, g)}{t_n} + \frac{1}{t_n}\inner{p -
p_n^\delta}{g-g^\delta} + \frac{\norm{p}^2}{2t_n^2} \\
& \leq \frac{G(p,g) - \inf_{q\in V} G(q, g)}{t_n} +
\frac{\norm{p-p_n^\delta}^2}{4t_n^2} + \delta^2+ \frac{\norm{p}^2}{2t_n^2}.
\end{eqnarray*}
This shows that
\begin{eqnarray}\label{rates:aux3}
\frac{\norm{p_n^\delta}^2}{8t_n^2} & \leq \frac{\norm{p-p_n^\delta}^2}{4t_n^2} +
\frac{\norm{p}^2}{4t_n^2} \\
& \leq \frac{G(p,g) - \inf_{q\in V} G(q, g)}{t_n}  + \delta^2+
\frac{3\norm{p}^2}{4t_n^2}.
\end{eqnarray}
Combining \eref{rates:aux3} with \eref{rates:aux2} and applying Lemma
\ref{dual:values} finally gives
\begin{eqnarray*}
\fl
\frac{1}{2}\norm{Ku_n^\delta - g^\delta}^2
 & + \frac{4}{t_n}D(u_n^\delta, u^\dagger) 
\leq\frac{1}{4}\Psi\left(\frac{16}{t_n}\right) + \frac{5\delta^2}{2} + \frac{8\norm{p_n^\delta}^2}{t_n^2}\\
&\le \frac{1}{4}\Psi\left(\frac{16}{t_n}\right) + 64\frac{G(p,g) - \inf_{q\in V} G(q, g)}{t_n}
+ \frac{133\delta^2}{2} + \frac{48\norm{p}^2}{t_n^2}.
\end{eqnarray*}
\end{proof}

\begin{lem}\label{rates:dualapprox}
Let Assumptions \ref{intro:mainass} and \ref{rates:varineq} hold. Then, 
\begin{equation*}
\inf_{p\in H}\biggl[ J^*(K^*p) + J(u^\dagger) - \inner{p}{g} +
\frac{\norm{p}^2}{2t}\biggr] \leq \frac{t}{2}\Psi\left(\frac{2}{t}\right).
\end{equation*}
\end{lem}

\begin{proof}
Classical duality theory (see~\cite[Chap III]{EkeTem76}) implies that
\[
\fl
\mu:=
\inf_{p\in H} \biggl[J^*(K^*p)+J(u^\dagger)-\inner{p}{g} + \frac{\norm{p}^2}{2t}\biggr]
= -\inf_{u \in X} \biggl[\frac{t}{2}\norm{Ku-g}^2 + J(u)-J(u^\dagger)\biggr],
\]
as the right hand side of this equation is the dual of the left hand side.
Using the variational inequality \eref{rates:varineqcond}
and the non-negativity of $D$,
we therefore find that
\begin{eqnarray*}
\mu &\le \sup_{u\in X}\biggl[\Phi\bigl(\norm{Ku-g}^2\bigr) - D(u,u^\dagger) - \frac{t}{2}\norm{Ku-g}^2\biggr]\\
&\le \sup_{u\in X}\biggl[\Phi\bigl(\norm{Ku-g}^2\bigr) - \frac{t}{2}\norm{Ku-g}^2\biggr].
\end{eqnarray*}
Replacing $\norm{Ku-g}^2$ by $s \ge 0$ in the last term
and using the definition of $\Psi$, we obtain
\[
\fl\mu 
\le \sup_{s \ge 0} \biggl[\Phi(s)-\frac{ts}{2}\biggr]
= \frac{t}{2}\sup_{s \ge 0} \biggl[\frac{2\Phi(s)}{t}-s\biggr]
= \frac{t}{2}\sup_{s \ge 0} \biggl[\frac{2s}{t}-\Phi^{-1}(s)\biggr]
= \frac{t}{2}\Psi\biggl(\frac{2}{t}\biggr),
\]
which proves the assertion.
\end{proof}

We close this section by a statement concerning the dual variables
$\set{p_1^\delta, p_2^\delta, \ldots}$ generated by the ALM. It is well known
(in the case when $\delta = 0$) that these stay bounded if and only if the
source condition \eref{dual:kkt} holds. Assumption \ref{rates:varineq},
however, allows to control their growth, as the following result shows.

\begin{cor}\label{rates:dualgrowth}
Let Assumptions \ref{intro:mainass} and \ref{rates:varineq} hold. Then, there exists a constant
$C>0$ such that
\begin{equation*}
\norm{p_n^\delta}^2\leq C t_n^2
\left(\Psi\left(\frac{2}{t_n}\right) + \delta^2\right)
\end{equation*}
\end{cor} 
\begin{proof}
It follows from \eref{rates:aux3} that there exists a constant $C>0$ such that
\begin{equation*}
\norm{p_n^\delta}^2 \leq C t_n \left(J^*(K^*p) + J(u^\dagger) - \inner{p}{g}  +
\frac{\norm{p}^2}{2t_n} + t_n \delta^2 \right)
\end{equation*}
for all $p\in H$. Applying Lemma \ref{rates:dualapprox} yields the desired
estimate. 
\end{proof}

\section{Morozov's Discrepancy Principle}\label{morozov}

In this section we study Morozov's discrepancy principle as an a posteriori
stopping rule for the ALM. To be more precise, if $\set{u_1^\delta,
u_2^\delta,\ldots}$ is generated by the ALM, Morozov's rule suggests to stop the
iteration at the index
\begin{equation}\label{morozov:rule}
n^*(\delta) = \min\set{n\in \N~:~\norm{Ku_n^\delta - g^\delta} \leq \rho
\delta},
\end{equation}
where $\rho>1$. In this section we prove convergence rates for the iterates
$u_{n^*(\delta)}^\delta$ given that Assumption \ref{rates:varineq} holds. 
Morozov's principle for the case when the source condition \eref{dual:kkt}
holds was studied in \cite{FriLorRes11}. Theorem \ref{morozov:mainthm} below 
extends this result to regularity classes that are delimited by the variational inequality
in Assumption \ref{rates:varineq}. Additionally to these, we will assume
\begin{ass}\label{morozov:ass}
Let Assumption \ref{rates:varineq} hold. 
\begin{enumerate}
  \item The mapping $s\mapsto \Phi(s)^2/s$ is non-increasing.
  \item The sequence of stepsizes $\set{\tau_1,\tau_2,\ldots}$ in the ALM is
  bounded.
\end{enumerate}
\end{ass}

\begin{thm}\label{morozov:mainthm}
  Let Assumptions \ref{intro:mainass} and \ref{morozov:ass} hold and assume that
  $n^*(\delta)$ is chosen according to Morozov's discrepancy principle
  \eref{morozov:rule} for some $\rho > 1$. Then there exists a constant $C >0$ independent of $\rho$
  such that
  \begin{equation*}
    D(u_{n^*(\delta)}^\delta, u^\dagger)\leq 
    \frac{C(\rho+1)^2\delta^2}{\Psi^{-1}\bigl((\rho^2-1)\delta^2\bigr)}
    + C(\rho+1)^2\delta^2\sup_{k\in\N}\tau_k.
  \end{equation*}  
\end{thm}

\begin{rem}
  Assume that the variational inequality~\eref{rates:varineqcond}
  is satisfied with $\Phi(s) = Cs^p$ for some $C > 0$ and $p > 0$.
  Then, setting $u = u^\dagger + tz$ for some $z \in X$
  and $t > 0$, the non-negativity of $D$
  implies in particular the inequality
  \[
  J(u^\dagger) - J(u^\dagger+tz) \le C t^{2p} \norm{Kz}^{2p}.
  \]
  Now assume that $p > 1/2$.
  Then we obtain, after dividing by $t$ and considering the
  limit $t \to 0^+$, that the directional derivative of $J$
  satisfies $-J'(u^\dagger)(z) \le 0$.
  Because $z$ was arbitrary, this implies that $u^\dagger$
  minimizes the regularization term $J$.
  Thus the variational inequality can hold in non-trivial situations,
  if and only if $p \le 1/2$.

  Now note that the same condition is required for the function
  $\Phi(s)^2/s = C^2s^{2p-1}$ to be non-increasing.
  Therefore, in the case of a variational inequality of H\"older type,
  Assumption~\ref{morozov:ass} imposes no relevant further restrictions
  on the index function.
\end{rem}

Before we give the proof of Theorem \ref{morozov:mainthm}, we state the
following Lemma, which is interesting in its own right. 

\begin{lem}\label{morozov:growthlem}
  Let Assumptions~\ref{intro:mainass} and~\ref{rates:varineq} hold and assume that
  $n^*(\delta)$ is chosen according to Morozov's discrepancy principle
  \eref{morozov:rule}. Then, 
  \begin{equation*}
    t_{n^*(\delta)} \leq \frac{2}{\Psi^{-1}((\rho^2-1)\delta^2)} +
    \tau_{n^*(\delta)}.
  \end{equation*}
\end{lem}

\begin{proof}
  Without loss of generality we may assume that $n^*(\delta) > 1$;
  else the assertion is trivial.
  Denote for the sake of simplicity $\bar n := n^*(\delta)-1$.
  Then it follows from \eref{morozov:rule} that
  $\norm{Ku_{\bar n}^\delta - g^\delta}^2 > \rho^2\delta^2$.
  Plugging in this relation into \eref{rates:aux0} yields
  \begin{eqnarray*}
    \fl
  \frac{\rho^2t_{\bar{n}}\delta^2}{2} + \frac{\norm{p-p_{\bar{n}}^\delta}^2}{2t_{\bar{n}}}
  &< \frac{t_{\bar{n}}}{2}\norm{Ku_{\bar{n}}^\delta-g^\delta}^2 + \frac{\norm{p-p_{\bar{n}}^\delta}^2}{2t_{\bar{n}}}\\
  &\le J^*(K^*p)-\inner{p}{g} + J(u^\dagger) + \frac{\norm{p}^2}{2t_{\bar{n}}} + \inner{p-p_{\bar{n}}^\delta}{g-g^\delta}
  \end{eqnarray*}
  for every $p \in H$.
  Applying Young's inequality
  \[
  \inner{p-p_{\bar{n}}^\delta}{g-g^\delta}
  \le \frac{\norm{p-p_{\bar{n}}^\delta}^2}{2t_{\bar{n}}} + \frac{t_{\bar{n}}\delta^2}{2},
  \]
  we obtain with Lemma~\ref{rates:dualapprox} the estimate
  \[
  \fl
  \frac{(\rho^2-1)t_{\bar{n}}\delta^2}{2}
  \le \inf_{p\in H}\biggl[J^*(K^*p)-\inner{p}{g} + J(u^\dagger) + \frac{\norm{p}^2}{2t_{\bar{n}}}\biggr]
  \le \frac{t_{\bar{n}}}{2}\Psi\biggl(\frac{2}{t_{\bar{n}}}\biggr).
  \]
  This proves that $(\rho^2-1)\delta^2 \leq \Psi(2\slash t_{\bar n})$.
  Now the assertion follows by applying the monotoneously increasing
  function $\Psi^{-1}$ to both sides of this inequality
  and adding the last step size $\tau_{n^*(\delta)}$.
\end{proof}

Next we need another lemma, which relates the condition on
$\Phi$ in Assumption~\ref{morozov:ass} to an equivalent condition
on the function $\Psi = (\Phi^{-1})^*$.

\begin{lem}\label{morozov:phigrowth}
  Let $\Phi$ be an index function and $\Psi$ the Fenchel conjugate of $\Phi^{-1}$.
  Then the mapping $s \mapsto \Phi(s)^2/s$ is non-increasing,
  if and only if the mapping $t \mapsto t^2 \Psi(2/t)$ is non-decreasing.
\end{lem}

\begin{proof}
  First note that, by means of the change of variables $t \mapsto 2/t$
  and ignoring the constant factor,
  the mapping $t \mapsto t^2 \Psi(2/t)$ is non-decreasing,
  if and only if the mapping $t \mapsto H(t) := \Psi(t)/t^2$ is non-increasing.
  Because $\Psi$ is convex and continuous, this condition is satisfied,
  if and only if $H'(t) \le 0$ for every $t > 0$ for which $\Psi'(t)$ exists.
  Now,
  \[
  H'(t) = \frac{\Psi'(t)}{t^2} - \frac{2\Psi(t)}{t^3} = \frac{1}{t^3}\bigl(t\Psi'(t)-2\Psi(t)\bigr),
  \]
  and therefore $H'(t) \le 0$ if and only if $t\Psi'(t)-2\Psi(t) \le 0$.
  Now recall that $\Psi$ is the Fenchel conjugate of $\Phi^{-1}$ and therefore
  $t\Psi'(t) = \Psi(t)+\Phi^{-1}\bigl(\Psi'(t)\bigr)$.
  Thus $H'(t) \le 0$, if and only if
  $\Phi^{-1}\bigl(\Psi'(t)\bigr)-\Psi(t) \le 0$.

  Similarly, the mapping $s \mapsto \Phi(s)^2/s$ is non-increasing,
  if and only if the mapping $s \mapsto \tilde{H}(s) := s^2/\Phi^{-1}(s)$ is non-increasing,
  which in turn is equivalent to the condition
  \[
  \tilde{H}'(s) = \frac{2s}{\Phi^{-1}(s)} - \frac{s^2{\Phi^{-1}}'(s)}{\Phi^{-1}(s)^2}
  = \frac{s\bigl(2\Phi^{-1}(s)-s{\Phi^{-1}}'(s)\bigr)}{\Phi^{-1}(s)^2} \le 0.
  \]
  Because of the equality $s{\Phi^{-1}}'(s) = \Phi^{-1}(s)+\Psi\bigl({\Phi^{-1}}'(s)\bigr)$,
  this is the case, if and only if $\Phi^{-1}(s)-\Psi\bigl({\Phi^{-1}}'(s)\bigr) \le 0$.
  The assertion now follows from the fact that
  $s = \Psi'(t)$ if and only if $t = {\Phi^{-1}}'(s)$,
  which, again, is a consequence of the fact that $\Phi^{-1}$ and $\Psi$ are conjugate.
\end{proof}

\begin{proof}[Proof of Theorem \ref{morozov:mainthm}]
  Throughout the proof we use the abbreviation $n = n^*(\delta)$.
  First observe that $K^* p_n^\delta \in \partial J(u_n^\delta)$ and thus 
  $J(u_n^\delta) - J(u^\dagger)\leq \inner{p_n^\delta}{Ku_n^\delta - g}$.
  From the discrepancy rule~\eref{morozov:rule} it follows that 
  \[
  \norm{Ku_n^\delta - g} \leq 
  \norm{Ku_n^\delta - g^\delta} + \delta
  \le (\rho + 1)\delta,
  \]
  and hence the variational
  inequality~\eref{rates:varineqcond} implies
  \begin{equation}\label{morozov:aux1}
    D(u_n^\delta, u^\dagger)\leq \norm{p_n^\delta} (\rho+1)\delta +
    \Phi((\rho+1)^2\delta^2).
  \end{equation}
  As in the proof of Lemma \ref{rates:genest} we observe that for all $s$, $r\geq 0$
  one has $s\Phi(r)\leq \Psi(s) + r$. Setting $r = (\rho+1)^2\delta^2$ and 
  $s = \Psi^{-1}((\rho^2-1)\delta^2)$, one finds, after dividing both sides of the
  inequality by $s$, that
  \begin{equation*}
    \fl
    \Phi\bigl((\rho+1)^2\delta^2\bigr) \leq
    \frac{(\rho^2-1)\delta^2}{\Psi^{-1}((\rho^2-1)\delta^2)} +
    \frac{(\rho+1)^2\delta^2}{\Psi^{-1}((\rho^2-1)\delta^2)}
    = \frac{2\rho(\rho+1)\delta^2}{\Psi^{-1}((\rho^2-1)\delta^2)},
  \end{equation*}    
  which yields an estimate for the second term in~\eref{morozov:aux1}.
  For estimating the first term,
  we note that Corollary~\ref{rates:dualgrowth} implies the estimate
  \[
  \norm{p_n^\delta} \le \tilde{C} t_n \biggl(\Psi\biggl(\frac{2}{t_n}\biggr) + \delta^2\biggr)^{1/2}
  \le \tilde{C} t_n\Psi\biggl(\frac{2}{t_n}\biggr)^{1/2} + \tilde{C}t_n\delta
  \]
  for some constant $\tilde{C} > 0$.
  By assumption, the mapping $x \mapsto \Phi(x)^2\slash x$ is non-increasing,
  and therefore, using Lemma~\ref{morozov:phigrowth}, the mapping
  $s \mapsto s^2\Psi(2\slash s)$ is non-decreasing.
  Thus we obtain, after using the estimate for $t_n$ of Lemma~\ref{morozov:growthlem}
  and the monotonicity of $\Psi$,
  \begin{eqnarray*}
  \norm{p_n^\delta}
  &\le \tilde{C} \biggl(\frac{2}{\Psi^{-1}((\rho^2-1)\delta^2)} + \tau_n\biggr)
  \Psi\biggl(\frac{2\Psi^{-1}\bigl((\rho^2-1)\delta^2\bigr)}{2+\tau_n\Psi^{-1}\bigl((\rho^2-1)\delta^2\bigr)}\biggr)^{1/2}\\
  &\qquad\qquad + \frac{2\tilde{C}\delta}{\Psi^{-1}\bigl((\rho^2-1)\delta^2\bigr)} + \tilde{C}\tau_n\delta\\
  &\le \frac{2\tilde{C}(\rho+1)\delta}{\Psi^{-1}\bigl((\rho^2-1)\delta^2\bigr)}
  + \tilde{C}\tau_n(\rho+1)\delta.
  \end{eqnarray*}
  Consequently we have
  \begin{eqnarray*} 
  D(u_n^\delta,u^\dagger)
  &\le \frac{2\tilde{C}(\rho+1)^2\delta^2}{\Psi^{-1}\bigl((\rho^2-1)\delta^2\bigr)}
  + \tilde{C}(\rho+1)^2\tau_n\delta^2 + \frac{2\rho(\rho+1)\delta^2}{\Psi^{-1}\bigl((\rho^2-1)\delta^2\bigr)}\\
  &\le \frac{2(\tilde{C}+1)(\rho+1)^2\delta^2}{\Psi^{-1}\bigl((\rho^2-1)\delta^2\bigr)}
  + \tilde{C}(\rho+1)^2\delta^2\sup_k \tau_k,
  \end{eqnarray*}
  which proves the assertion with $C := 2(\tilde{C}+1)$.
\end{proof}

\section{Examples}\label{ex}

In this section we discuss particular instances of the variational
inequality \eref{rates:varineqcond} and the implications of the general results
in Sections \ref{rates} and \ref{morozov} for these special scenarios. The first
two examples shed some light on the relation of variational inequalities and
more standard notions of source conditions: the KKT condition \eref{dual:kkt}
and H{\"o}lder-type conditions. The third example shows an example from sparsity
promoting regularization, where standard notions of source conditions together
with an additional restricted injectivity assumption allow the derivation of
convergence rates with respect to norm instead of the Bregman distance.

\subsection{Standard Source Condition}\label{ex:standard}

It is quite easy to see that the standard source condition \eref{dual:kkt}
implies the variational inequality \eref{rates:varineqbreg}. Indeed, assume
that $u^\dagger$ is a solution of \eref{intro:primal} and that $K^*p^\dagger\in
\partial J(u^\dagger)$ for some $p^\dagger \in H$. By defining $\xi^\dagger =
K^*p^\dagger$ one observes
\begin{equation*}
\inner{\xi^\dagger}{u^\dagger - u}_{X^*,X} = \inner{p^\dagger}{g - Ku} \leq
\norm{p^\dagger}\norm{Ku-g}.
\end{equation*}
Setting $\beta= 1$ and $\Phi(t) = \norm{p^\dagger} t^{1\slash 2}$ gives
\eref{rates:varineqbreg}.

The converse is in general not true, i.e., \eref{rates:varineqbreg} with
$\Phi(t) = \gamma t^{1\slash 2}$ ($\gamma > 0$) does not imply the existence of
a $p^\dagger\in V$ such that $K^*p^\dagger\in \partial J(u^\dagger)$.  However,
if \eref{rates:varineqbreg} is replaced by the stronger condition
\begin{equation}\label{rates:eqnmod}
\inner{\xi^\dagger}{u^\dagger-u}_{X^*,X}\leq (1-\beta) D_J(u,u^\dagger) +
\gamma \norm{Ku - g},
\end{equation}
for all $u\in X$, then the two notions are equivalent.
Here, $D_J(u,v) = J(u) - J(v) - J'(v)(u-v)$ and $J'(v)(w)$ is the directional derivative of $J$ at
$v$ in direction $w$:
\begin{equation*}
J'(v)(w) = \lim_{h\ra 0^+} \frac{1}{h}(J(v+hw) - J(v)). 
\end{equation*}
Note that for convex $J$, the directional derivative is well-defined for every
$v$ and $w$ (though it takes values in $[-\infty, \infty]$) and is positively
one-homogeneous, i.e. $J'(v)(tw) = t J'(v)(w)$ for all $t > 0$. 

In order to see the aforementioned equivalence, let
$v\in X$ and set $u = u^\dagger - t v$ in \eref{rates:eqnmod} for some $t > 0$. Then, 
\begin{equation*} 
\inner{\xi^\dagger}{tv}_{X^*,X} \leq (1-\beta) D_J(u^\dagger - tv, u^\dagger) +
\gamma\norm{tKv}.
\end{equation*}
Since the mapping $w\mapsto J'(u^\dagger)(w)$ is positively one-homogeneous, this implies
that 
\begin{equation*}
\inner{\xi^\dagger}{v}_{X^*,X} \leq (1-\beta)\left(\frac{J(u^\dagger - tv) -
J(u^\dagger)}{t} - J'(u^\dagger)(-v) \right) + \gamma \norm{Kv},
\end{equation*}
for all $v\in X$ and $t>0$. Letting $t\ra 0^+$ this shows that
$\inner{\xi^\dagger}{v}_{X^*,X}\leq \gamma\norm{Kv}$ for all $v\in X$ and hence
$K^*p^\dagger = \xi^\dagger$ for  some $p^\dagger \in H$ according to \cite[Lem. 8.21]{SchGraGroHalLen09}.

In the particular case where the mapping $J$ is G\^ateaux differentiable
at $u^\dagger$, the subdifferential $\partial J(u^\dagger)$ contains
a single element $\xi^\dagger$, which coincides with the directional derivative,
that is, $\inner{\xi^\dagger}{v} = J'(u^\dagger)(v)$ for every $v \in X$.
Thus, in this case, the source condition is equivalent with the
variational inequality.

If $\Phi(t) = \gamma t^{1\slash 2}$ then the Fenchel conjugate $\Psi$ of
$\Phi^{-1}$ reads as $\Psi(t) = \gamma\slash (2\sqrt{2}) t^2$. Hence it follows
from Theorem \ref{rates:mainthm} that there exists a constant $C>0$ such that
\begin{equation*}
D_J^{K^*p^\dagger}(u_n^\delta, u^\dagger) \leq C \delta
\end{equation*}
given the a priori stopping rule $t_n \asymp \delta^{-1}$. This is the well
known convergence rate result for the standard source condition (see
\cite{BurResHe07,FriSch10}). We note that the results in \cite{FriSch10} are
slightly stronger, as they give $\delta$-rates for the \emph{symmetric}
Bregman distance (see also \cite{FriLorRes11}). If Morozov's discrepancy
principle \eref{morozov:rule} is applied as an a posteriori stopping rule, we
obtain from Theorem \ref{morozov:mainthm} that
\begin{equation*}
D_J^{K^*p^\dagger}(u_{n^*(\delta)}^\delta, u^\dagger) \leq
C\sqrt{\frac{(\rho+1)^3 }{\rho-1}}\delta + C (\rho+1)^2
\delta^2\sup_{k\in\N}\tau_k.
\end{equation*}
This coincides with the results in \cite[Thm. 4.3]{FriLorRes11}, where Morozov's
discrepancy rule for the standard source condition was studied.

\subsection{H{\"older}-type Conditions}\label{ex:hoelder}

In this section we study the relationship between the variational inequality
\eref{rates:varineqbreg} and H{\"older}-type source conditions for the
iteration \eref{intro:itertik}. 

We first consider the case of the \emph{iterated Tikhonov method}, i.e., $ L =
\id$ and thus $J(u) = {1\over 2}\norm{u}^2$. Then, a solution
$u^\dagger$ of \eref{intro:primal} is said to satisfy a H{\"older} condition with exponent $0\leq \nu< {1\over 2}$, if
$(K^*K)^\nu p^\dagger = u^\dagger = \partial J(u^\dagger)$. 
If $u^\dagger$ satisfies a H{\"o}lder condition with exponent $\nu$, then
\eref{rates:varineqbreg} holds with $D_J^{u^\dagger}(u,u^\dagger) =
{1\over 2}\norm{u-u^\dagger}^2$ and $\Phi(s) \asymp s^\frac{2\nu}{1+2\nu}$. To
see this, observe that the interpolation inequality (cf. \cite[p.47]{EngHanNeu96}) implies
\begin{eqnarray*}
\inner{u^\dagger}{u^\dagger - u} & \leq \norm{p^\dagger}\norm{(K^*K)^\nu
(u^\dagger - u)} \\
& \leq \norm{p^\dagger} \norm{(K^*K)^\frac{1}{2}(u^\dagger -
u)}^{2\nu} \norm{u^\dagger - u}^{1-2\nu} \\
& = 2^{{1\over 2} -\nu} \norm{p^\dagger} \bigl(\norm{Ku - g}^2\bigr)^\nu
D_J^{u^\dagger}(u,u^\dagger)^{1-2\nu \over 2}.
\end{eqnarray*} 
Using Young's inequality $ab\leq a^p\slash p + b^q\slash q$ with $q = 2\slash
(1-2\nu)$ and $p = 2\slash (1+2\nu)$ shows for all $\eta > 0$
\begin{eqnarray*} 
\fl\bigl(\norm{Ku - g}^2\bigr)^\nu D_J^{u^\dagger}(u,u^\dagger)^{1-2\nu \over 2}
& = \frac{1}{\eta}\bigl(\norm{Ku - g}^2\bigr)^\nu
\eta D_J^{u^\dagger}(u,u^\dagger)^{1-2\nu \over 2}  \\ 
& = \frac{1+2\nu}{2\eta^{2\over(1+2\nu)}}(\norm{Ku - g}^2)^{\frac{2\nu}{1+2\nu}} +
\frac{\eta^{2\over(1-2\nu)}(1-2\nu)}{2}D_J^{u^\dagger}(u,u^\dagger).
\end{eqnarray*}
Choosing $\eta$ such that $1-\beta = \eta^{2\over
1-2\nu}\norm{p^\dagger}({1-2\nu\over 2})2^{1-2\nu\over 2} < 1$ results in
\eref{rates:varineqbreg} after setting $\Phi(s) = c s^\frac{2\nu}{1+2\nu}$ with 
$c= {1+2\nu\over 2\eta^{2\slash(1+2\nu)}}\norm{p^\dagger}2^{1-2\nu\over 2}$.

In case of the \emph{iterated Tikhonov-Morozov method}, we consider 
\eref{intro:itertik} with $K=\id$ and $L\colon D(L)\subset X\ra \tilde H$ being a
densely defined, closed linear operator. Recall that in this case $\hat L =
(\id + LL^*)^{-1}$ and $\tilde L = (\id + L^*L)^{-1}$ are self-adjoint and bounded
linear operators (cf. \cite[Chap. 2.4]{Gro07}). A solution $u^\dagger$ of
\eref{intro:primal} is said to satisfy a H{\"o}lder condition with exponent $0\leq \nu\leq
{1\over 2}$ if $Lu^\dagger = \hat L^\nu \omega^\dagger$ for some
$\omega^\dagger\in \tilde H$. We show that this condition implies
\eref{rates:varineqcond} when $D(u,u^\dagger)$ equals ${\gamma\over
2}\norm{Lu - Lu^\dagger}^2$ (for some $\gamma\in(0,1)$) whenever $u\in D(L)$ and
$+\infty$ else. To see this, recall that $J(u) = \infty$ if $u\not\in D(L)$.
Thus \eref{rates:varineqcond} is equivalent to 
\begin{equation}\label{ex:hoelaux1}
\inner{Lu^\dagger}{Lu^\dagger - Lu} \leq (1-\gamma)\norm{Lu - Lu^\dagger}^2 +
\Phi(\norm{u - u^\dagger}^2)
\end{equation}
for all $u\in D(L)$. Setting $Lu^\dagger=\hat L^\nu \omega^\dagger$ shows
together with the interpolation inequality and \cite[Lem. 2.10]{Gro07} that for
all $u\in D(L)$
\begin{eqnarray*}
\inner{Lu^\dagger}{Lu^\dagger - Lu} & = \inner{\omega^\dagger}{\hat
L^{\nu}(Lu^\dagger - Lu)} \\ 
& \leq \norm{\omega^\dagger}\norm{\hat L^{1\over 2}(Lu^\dagger -
Lu)}^{2\nu}\norm{Lu^\dagger - Lu}^{1-2\nu} \\
& \leq \norm{\omega^\dagger} \bigl\| L \tilde L^{1\over 2}\bigr\|^{2\nu}
\norm{u^\dagger - u}^{2\nu}\norm{Lu^\dagger - Lu}^{1-2\nu}.
\end{eqnarray*}
With the same arguments as in the case of the iterated Tikhonov method above, we
conclude that \eref{ex:hoelaux1} holds with $\Phi(s) = \tilde c s^{2\nu \over
2\nu +1}$ for some constant $\tilde c > 0$. 

Now let again be $X$ a general Banach space and $J\colon X\ra \Rb$ be convex such that
Assumptions \ref{intro:mainass} are satisfied. As revealed by the calculations
above, the variational inequality \eref{rates:varineqcond} with $\Phi(s) \asymp
s^\frac{2\nu}{1+2\nu}$ can be seen as a generalized H{\"o}lder condition. Note,
that in this case the Legendre conjugate $\Psi$ of $\Phi^{-1}$ comes as
$\Psi(t)\asymp t^{1+2\nu}$ and thus Theorem \ref{rates:mainthm} amounts to say
that there exists a constant $C>0$ such that
\begin{equation*}
D(u_n^\delta, u^\dagger) \leq C \delta^{4\nu\over 1+2\nu}
\end{equation*}
if $t_n \asymp \delta^{-2\over 1+2\nu}$ and Morozov's discrepancy principle
\eref{morozov:rule} shows that
\begin{equation*}
D(u_{n^*(\delta)}^\delta, u^\dagger) \leq C
\left(\frac{(\rho+1)^{1+4\nu}}{\rho-1}\right)^{1\over 1+2\nu} \delta^{4\nu\over
1+2\nu} + C(\rho+1)^2\delta^2 \sup_{k\in\N}\tau_k.
\end{equation*}
These results coincide with the lower order rates for the iterated Tikhonov
method \cite{HanGro98} and iterated Tikhonov-Morozov method \cite{Gro07}.

\subsection{Sparsity Promoting Regularization}\label{ex:sparse}

We now discuss the application of the results derived in this paper
to sparsity promoting regularization.
To that end, we assume that $X$ is a Hilbert space with orthonormal basis 
$\set{\phi_i : i \in \N}$,
and we consider the regularization term 
$J(u) := \sum_i \abs{\inner{\phi_i}{u}}^q$ for some $1 \le q < 2$
(see~\cite{DauDefDem04}).
In~\cite{GraHalSch08}, it has been shown that, for Tikhonov regularization,
this setting allows the derivation of convergence rates of order $\mathcal{O}(\delta^q)$
with respect to the norm, if $u^\dagger$ satisfies the standard source
condition $K^*p^\dagger \in \partial J(u^\dagger)$ for some $p^\dagger \in H$,
and, additionally, a \emph{restricted injectivity condition} holds.
In the following, we will generalize these results to
the Augmented Lagrangian Method and source conditions of H\"older type.

Assume that there exists $0 < \nu \le 1/2$ such that 
$(K^*K)^\nu p^\dagger = \xi^\dagger \in \partial J(u^\dagger)$
and that $\supp(u^\dagger) := \set{i\in\N : \inner{\phi_i}{u} \neq 0}$ is finite.
In case $q > 1$ assume in addition that the restriction of $K$ to
$\spa\set{\phi_i : i \in \supp(x^\dagger)}$, 
and in case $q = 1$ assume that the restriction of $K$ to
$\spa\set{\phi_i : \abs{\inner{\phi_i}{\xi^\dagger}} < 1}$ is injective.
We will show in the following that, under these assumptions,
there exists a constant $C > 0$ such that~\eref{rates:varineqcond} holds with 
$D(u,u^\dagger) = C\norm{u^\dagger-u}^q$ and $\Phi(s)\asymp s^{\frac{q\nu}{q-1+2\nu}}$
in case $q > 1$,
and with $D(u,u^\dagger) = C\norm{u^\dagger-u}$ and $\Phi(s) \asymp s^{\frac{1}{2}}$
for $q = 1$.

It has been shown in~\cite[Proofs of Thms.~13, 15]{GraHalSch08}
that the given assumptions imply the existence of constants $C_1$, $C_2 > 0$
such that
\[ 
C_1\norm{u^\dagger-u}^q \le C_2\norm{Ku-g}^q + J(u)-J(u^\dagger)-\inner{\xi^\dagger}{u-u^\dagger}
\] 
for all $u \in X$.
Applying the interpolation inequality to $\inner{\xi^\dagger}{u-u^\dagger}$, we obtain,
similarly as in Section~\ref{ex:hoelder},
the estimate
\begin{eqnarray*}
\fl
C_1\norm{u^\dagger-u}^q &\le C_2\norm{Ku-g}^q + J(u)-J(u^\dagger)
+ \norm{p^\dagger}\norm{Ku-g}^{2\nu}\norm{u^\dagger-u}^{1-2\nu}.
\end{eqnarray*}
Now Young's inequality with $p = q/(1-2\nu)$ and $p_*=q/(q-1+2\nu)$ shows that
\begin{eqnarray*} 
  \fl
  \Bigl[C_1 -\norm{p^\dagger}\frac{1-2\nu}{q}\eta^{\frac{q}{1-2\nu}}\Bigr]\norm{u^\dagger-u}^q
  &\le C_2 \norm{Ku-g}^q + J(u)-J(u^\dagger)\\
  &\quad + \norm{p^\dagger}\frac{q-1+2\nu}{q}\eta^{\frac{q}{q-1+2\nu}} \norm{Ku-g}^{\frac{2\nu q}{q-1+2\nu}}.
\end{eqnarray*}
Choosing $\eta > 0$ such that $C = C_1 - \norm{p^\dagger}\frac{1-2\nu}{q}\eta^{\frac{q}{1-2\nu}} > 0$
and setting 
\[
\Phi(s) = C_2 s^{\frac{q}{2}} + \norm{p^\dagger}\frac{q-1+2\nu}{q}\eta^{\frac{q}{q-1+2\nu}}s^{\frac{2\nu q}{q-1+2\nu}},
\]
we obtain the variational inequality~\eref{rates:varineqcond}.
Because $\frac{2\nu q}{q-1-2\nu} \le q$, the asymptotic behaviour
of $\Phi$ for $s \to 0$ is governed by its second term,
which shows that $\Phi(s) \asymp s^{\frac{q\nu}{q-1+2\nu}}$.
Moreover, in the special case $q = 1$,
the term $s^{\frac{q\nu}{q-1+2\nu}}$ reduces to $s^{\frac{1}{2}}$
independent of the type of the source condition.
For the function $\Psi$, we obtain the asymptotic behaviour
$\Psi(s) \asymp s^{\frac{q-1+2\nu}{q-1+(2-q)\nu}}$.
Thus, Theorem \ref{rates:mainthm} shows that for
$t_n \asymp \delta^{-2\frac{q-1+(2-q)\nu}{q-1+2\nu}}$ we have the estimate
\[
\norm{u_n^\delta - u^\dagger} \leq C\delta^{\frac{2\nu}{q-1+2\nu}}
\]
for $\delta > 0$ sufficiently small, and a similar estimate
for Morozov's discrepancy principle.

\begin{rem}
  In~\cite{GraHalSch11}, it has been shown for Tikhonov regularization
  with $J(u) = \sum_i \abs{\inner{\phi_i}{u}}$,
  which is the special case of the ALM with a single iteration step,
  that a linear convergence rate with respect to the norm
  is equivalent to the usual source condition.
  Thus the results above imply that, in the case $q = 1$,
  the H\"older type source condition $(K^*K)^\nu p^\dagger \in \partial J(u^\dagger)$
  in fact already implies the standard source condition
  $K^* \tilde{p}^\dagger \in \partial J(u^\dagger)$
  for some different source element $\tilde{p}^\dagger$.
\end{rem}
 
\ack 

The second author would like to thank Axel Munk and the staff of the Institute
for Mathematical Stochastics at the University of G\"ottingen for their
hospitality during his stay in G\"ottingen. This work was partially funded by
the DFG-SNF Research Group FOR916 \emph{Statistical Regularization and
Qualitative Constraints} (Z-Project).

          
\section*{References}

\bibliographystyle{jphysicsB}              
\bibliography{literature}  

\end{document}